\documentclass[11pt]{amsart}
\usepackage[utf8]{inputenc}
\usepackage[normalem]{ulem}
\usepackage{mathtools}
\usepackage{amssymb, amsmath, enumerate, amsthm, cite, colonequals, stmaryrd, lipsum, graphicx, fancyvrb,fancyhdr, indentfirst, setspace}
\usepackage[margin=1in]{geometry}
\usepackage{mathtools}
\makeatletter
\@namedef{subjclassname@2020}{\textup{2020} Mathematics Subject Classification}
\makeatother

\usepackage{hyperref}
\hypersetup{pdfpagemode=UseNone}
\newcommand{\prj}{\operatorname{pr}}
\newcommand{\res}{\operatorname{res}}
\newcommand{\NestHilb}{X^{[n,n+1]}}
\newcommand{\diff}{\operatorname{diff}}

\usepackage[capitalise]{cleveref}
\crefformat{equation}{(#2#1#3)}
\crefrangeformat{equation}{(#3#1#4--#5#2#6)}
\crefformat{enumi}{(#2#1#3)}
\crefrangeformat{enumi}{(#3#1#4--#5#2#6)}

\usepackage[pagewise]{lineno}
\let\oldequation\equation
\let\oldendequation\endequation
\renewenvironment{equation}{\linenomathNonumbers\oldequation}{\oldendequation\endlinenomath}
\expandafter\let\expandafter\oldequationstar\csname equation*\endcsname
\expandafter\let\expandafter\oldendequationstar\csname endequation*\endcsname
\renewenvironment{equation*}{\linenomathNonumbers\oldequationstar}{\oldendequationstar\endlinenomath}
\let\oldalign\align
\let\oldendalign\endalign
\renewenvironment{align}{\linenomathNonumbers\oldalign}{\oldendalign\endlinenomath}
\expandafter\let\expandafter\oldalignstar\csname align*\endcsname
\expandafter\let\expandafter\oldendalignstar\csname endalign*\endcsname
\renewenvironment{align*}{\linenomathNonumbers\oldalignstar}{\oldendalignstar\endlinenomath}

\usepackage{tikz-cd}




\newcounter{intro}
\newcounter{result}

\theoremstyle{definition}
\newtheorem{theorem}[result]{Theorem}
\newtheorem{lemma}[result]{Lemma}
\newtheorem{result}[result]{result}

\newtheorem{proposition}[result]{Proposition}
\newtheorem{corollary}[result]{Corollary}
\newtheorem{definition}[result]{Definition}

\newtheorem{remark}[result]{Remark}

\newtheorem{notation}[result]{Notation}

\numberwithin{equation}{subsection}

\title[Nef cones of Hilbert schemes of points on some K3 surfaces]{Nef cones of Hilbert schemes of points on some K3 surfaces}

\author[U.~Dutta]{Uttaran Dutta}
\address{Department of Mathematics, University of South Carolina}
\email{udutta@email.sc.edu}

\author[S.~Edwards]{Sean Edwards}
\address{Department of Mathematics, Statistics and Computer Science, 
University of Illinois Chicago}
\email{smedwar2@uic.edu}

\author[N.~Raha]{Neelarnab Raha}
\address{Department of Mathematics, The Pennsylvania State University}
\email{neelraha@psu.edu}

\date{\today}

\thanks{The first author was supported by the National Science Foundation under Award No. 2302263 and the Simons Dissertation Fellowship in Mathematics. The second author was partially supported by NSF DMS 2037569}

\DeclareMathOperator{\ch}{ch}

\DeclareMathOperator{\Pic}{Pic}

\DeclareMathOperator{\Supp}{Supp}

\DeclareMathOperator{\sHom}{\mathcal{H}\kern -.5pt\mathit{om}}
\DeclareMathOperator{\sTor}{\mathcal{T}\kern -1.5pt\mathit{or}}

\DeclareMathOperator{\Nef}{Nef}

\begin{document}

\begin{abstract}
    We illustrate the typical usage of Bayer and Macr\`{i}'s Positivity Lemma to compute the nef cones of the Hilbert schemes $X^{[n]}$ by combining the Bridgeland stability methods (for large $n$) and classical methods (for small $n$). We use Mori dream K3 surfaces $X$ of Picard rank $2$ as our working example. We also compute the nef cones of the nested Hilbert schemes $X^{[n,n+1]}$ for such $X$, for large $n$.
\end{abstract}

\maketitle
\setcounter{tocdepth}{1}
\numberwithin{result}{section}
\tableofcontents

\section{Introduction}\label{sec:Intro}
Given a projective variety $Y$, the ample cone in its Néron–Severi space encodes how $Y$ embeds into projective space. Understanding this cone is a key step in running the minimal model program on $Y$. The nef cone, $\Nef(Y)$, is the closure of the ample cone and serves as a fundamental birational invariant. However, computing $\Nef(Y)$ explicitly is often difficult. One has to first prove that a given divisor is nef, and then show that the nef cone \emph{does not extend further}. The latter typically requires producing curves orthogonal to the candidate boundary divisors. For the former, one often pulls back known nef divisors along suitable morphisms. A fairly recent technique involving Bridgeland stability conditions developed by Bayer and Macr\`{i} (see \cite[Lemma 3.3]{Bayer-Macri-PositivityLemma}) and several other authors (see \cite{Bol+:2016}) has been used in \cite{Bol+:2016,Kopper:21} to compute the nef cones of Hilbert schemes of points on hypersurfaces in $\mathbb{P}^3$ and on blow-ups of $\mathbb{P}^2$ for example.

Hilbert schemes of points form an important class of parameter spaces, and in this article we explore the question of finding the nef cones of these Hilbert schemes on K3 surfaces. As a corollary, we also obtain information about the \emph{nested} Hilbert schemes.

\subsection{What is known}

In \cite{Bayer-Macri-PositivityLemma}, Bayer and Macr\`{i} develop a powerful tool (called the \emph{Positivity Lemma}) to create nef divisor classes on moduli stacks of Bridgeland semistable objects on smooth projective varieties. They also explain when a curve class is orthogonal to such a nef divisor class. They then use this technique to compute the nef cone of the Hilbert schemes $X^{[n]}$ of $n$ points on Picard rank 1 K3 surfaces $X$ for large $n$. See \cite[Lemma 3.3, Proposition 10.3]{Bayer-Macri-PositivityLemma}. In fact, their methods (see \cite[Theorem 12.1]{Bayer-Macri-K3}) compute the nef cone of the Hilbert scheme $X^{[n]}$ for any $n$ for any K3 surface $X$. Bridgeland stability techniques are also utilized by Li and Zhao (\hspace{1sp}\cite{LiZhao19}) on $X=\mathbb{P}^2$, by Yanagida and Yoshioka (\hspace{1sp}\cite{YY,Yoshioka16}) on abelian surfaces $X$, and by Nuer and Yoshioka \cite{Nuer},\cite{NuerYoshioka} on Enriques surfaces $X$ to answer similar questions. See also \cite{ABCH13}.

Based on the Positivity Lemma, a general technique for constructing (extremal) nef divisor classes on $X^{[n]}$ (for large enough $n$) is discussed in \cite{Bol+:2016}. Using this technique, Kopper (see \cite{Kopper:21}) has studied the nef cones $X^{[n]}$ (with $n\geq3$) for rational elliptic surfaces $X$. The main idea is to consider a certain cone $\Lambda$ consisting of divisor classes that are non-negative on some particular curves on $X^{[n]}$, so that $\Nef(X^{[n]})\subseteq\Lambda$. Then he bounds $\Lambda$ by a larger cone, and utilizes that to show that the extremal divisors of $\Lambda$ are actually nef (using the method from \cite{Bol+:2016} and the Weyl group action on $\operatorname{NE}(X)$). For small $n$, one usually needs to utilize more classical techniques, using the geometry of the surface in consideration in much more depth. See \cite{Raha2024} for example, where the author studies the nef cones of $X^{[n]}$ for hypersurfaces $X$ in $\mathbb{P}^3$, with $n$ small.

For K3 surfaces, however, several favorable properties hold, allowing one to eliminate the “large enough $n$” requirement. In particular, \cite[Theorem 5.7]{Bayer-Macri-K3} provides a detailed description of the behavior of walls and divisorial contractions, many of which are governed by solutions to certain Pell’s equations. Together with \cite[Theorem 1.2]{Bayer-Macri-K3}, this makes it possible to compute the movable cone of $X^{[n]}$ for a K3 surface $X$. Moreover, \cite[Theorems 12.1 and 12.3]{Bayer-Macri-K3} allow one to compare the nef and the movable cones of $X^{[n]}$.

Using the computations of the nef cones for Hilbert schemes of points on surfaces, \cite{Ryan-Yang:2020} studies the natural projections $X^{[n,n+1]} \to X^{[n]}$ and $X^{[n, n+1]} \to X^{[n+1]}$ of the nested Hilbert schemes $X^{[n,n+1]}$, and the natural curve and divisor classes arising from pullbacks along these projections to compute the nef cones of the nested Hilbert schemes.

\subsection{What we do in this paper}

Our primary goal in this primarily expository article is to illustrate the \emph{typical} usage of Bayer and Macr\`{i}'s Positivity Lemma and the methods developed in \cite{Bol+:2016} to compute the nef cones of the Hilbert schemes $X^{[n]}$ by combining the Bridgeland stability methods (for large $n$) and classical methods (for small $n$). We use K3 surfaces $X$ of Picard rank $2$ that are also Mori dream surfaces as our working example. However, we would like to remind the reader that for K3 surfaces, one can simply utilize the strategy mentioned earlier (more specifically, \cite[Theorem 12.1]{Bayer-Macri-K3}) to get the desired result for any $n$.

We need to briefly introduce some notation before stating the first result. Mori dream K3 surfaces of Picard rank $2$ have divisor class group $\operatorname{Cl}(X)\cong\mathbb{Z}w_1\oplus\mathbb{Z}w_2$ for some smooth curves $w_1$ and $w_2$ on $X$ that are either rational or elliptic. We use the following helpful terminology: \begin{enumerate}
	\item If both $w_1$ and $w_2$ are elliptic, we say that the K3 surface belongs to \emph{case I}.
	\item If one of $w_1$ and $w_2$ is elliptic and the other one is rational, we say that the K3 surface belongs to \emph{case II}.
	\item If both $w_1$ and $w_2$ are rational, we say that the K3 surface belongs to \emph{case III}.
\end{enumerate} In any case, we set $k:=w_1\cdot w_2$. If $X$ is a smooth projective surface with irregularity $q(X)=0,$ then $N^1(X^{[n]})$ is spanned by the divisor $B^{[n]}$ of nonreduced schemes and divisors $L^{[n]}$ induced by divisors $L\in\operatorname{Pic}(X)$. Given any curve $C$ on $X$ admitting a $g^1_n$, the fibers of the map $C\to\mathbb{P}^1$ define a rational curve on $X^{[n]}$, which we denote by $C_{[n]}$. We shall abuse notation and denote its class also by $C_{[n]}$. If $C$ is $w_i$ ($i=1,2$), we write $w_{i,[n]}$ to denote the corresponding $C_{[n]}$. We let $C_0$ be the class of a curve contracted by the Hilbert-Chow morphism. Finally, we let $\Lambda\subset N^1(X^{[n]})$ be the cone of all divisors on $X^{[n]}$ that have non-negative intersections with $w_{1,[n]},w_{2,[n]}$ and $C_0$. We are now ready to state our results.

\begin{theorem}[Theorem \ref{thm:Nef_cone_case1} and Theorem \ref{thm:Nef_cone_case2and3}]
	The nef cone of the Hilbert scheme $X^{[n]}$ can be described as follows:
	\begin{enumerate}
		\item Case I: For any integer $n\geq9k/8$, the nef cone of the Hilbert scheme $X^{[n]}$ is the cone $\Lambda$ constructed above. Dually, the Mori cone $\overline{\operatorname{NE}}(X^{[n]})$ is the cone in $N_1(X^{[n]})$ spanned by the curve classes $w_{1,[n]},w_{2,[n]}$ and $C_0$.
		
		\item Case II: For any large enough integer $n$, the nef cone of the Hilbert scheme $X^{[n]}$ is the cone $\Lambda$ constructed above. Dually, the Mori cone $\overline{\operatorname{NE}}(X^{[n]})$ is the cone in $N_1(X^{[n]})$ spanned by the curve classes $w_{1,[n]},w_{2,[n]}$ and $C_0$.
		
		More precisely, the above assertions hold whenever $n$ satisfies the inequalities $n>k$ and $$32(k^2+2k+1)n^4-4k(25k^2+41k+16)n^3+4k^2(26k^2+33k+7)n^2-4k^3(9k^2+8k-1)n-k^4\geq0.$$
		
		\item Case III: Assume $k\geq 3$. For any large enough integer $n$, the nef cone of the Hilbert scheme $X^{[n]}$ is the cone $\Lambda$ constructed above. Dually, the Mori cone $\overline{\operatorname{NE}}(X^{[n]})$ is the cone in $N_1(X^{[n]})$ spanned by the curve classes $w_{1,[n]},w_{2,[n]}$ and $C_0$.
		
		More precisely, the above assertions hold whenever $n$ satisfies the inequalities $n\geq k\geq 3$ and $$8n^2+(2-9k)n+8\geq0.$$
	\end{enumerate}
\end{theorem}

The above calculation can now be used to obtain a bound for the nef cone of $X^{[n]}$ for small $n$.

\begin{proposition}[Proposition \ref{prop:bound_for_small_n_case1}]
    Let $X$ be a K3 surface that is in case I, and suppose $n'$ is the positive integer for which $n'\geq9k/8$ and $n'-1<9k/8$. Let $n$ be a positive integer with $1<n<n'$, and $j:=n'-n$. Then the nef cone of $X^{[n]}$ is sandwiched between the cones $\Lambda$ and $\Lambda'$, where $\Lambda$ is the cone in $N^1(X^{[n]})$ spanned by $\Nef(X)$ and $$\frac{n}{k}w_1^{[n]}+\frac{n}{k}w_2^{[n]}-\frac{1}{2}B^{[n]},$$ and $\Lambda'$ is the cone in $N^1(X^{[n]})$ spanned by $\Nef(X)$ and the divisor $$\frac{n+j}{k}w_1^{[n]}+\frac{n+j}{k}w_2^{[n]}-\frac{1}{2}B^{[n]}.$$
\end{proposition}

In the particular case of $n=k=2$, we study the geometry of the K3 surface $X$ closely. We consider a $2$-to-$1$ cover of a smooth quadric $Q$ in $\mathbb{P}^3$ by $X$ (where $X$ is in case I), and look at nef divisor classes pulled back along the induced rational map $X^{[2]}\dashrightarrow Q^{[2]}$. Studying curves outside the domain of definition of this rational map, we obtain the following result.

\begin{theorem}[Theorem \ref{thm:Nef_cone_nk2}]
    In case I when $k=2$, the nef cone of the Hilbert scheme $X^{[2]}$ is spanned by $w_1^{[2]}, w_2^{[2]}$ and $$w_1^{[2]} + w_2^{[2]} - \frac{1}{2}B^{[2]}.$$
\end{theorem}

In subsection \ref{subsec:nef_cone_nk2_lattice_computation}, we once again calculate the nef cone of the Hilbert scheme $X^{[2]}$ in case I when $k=2$, this time directly using \cite[Theorem 12.1]{Bayer-Macri-K3}. In other words, we bypass \cite{Bol+:2016} and illustrate a direct computation of the nef cone of $X^{[n]}$ in the case of K3 surfaces $X$.

Finally, analogous to what is done in \cite{Ryan-Yang:2020}, we compute the nef cones of the nested Hilbert schemes $X^{[n,n+1]}$ (for $n\gg0$) of points on K3 surfaces discussed above. In this case, we restrict to Case I. The intersection table \ref{nestedtable} is the key ingredient. Because of these intersection numbers, the argument of \cite{Ryan-Yang:2020} carries through with no modification when $X$ is of case I:

\begin{corollary}
    Let $X$ be in case I. Let $D_n$ be the extremal ray of $\mathrm{Nef}(X^{[n]})$ which does not come from the inclusion $\mathrm{Nef}(X) \hookrightarrow \mathrm{Nef}(X^{[n]})$. Then the nef cone of the nested Hilbert scheme $X^{[n,n+1]}$ for $n\geq9k/8$ is spanned by 
    \[
    \{\mathrm{pr}_a^*(w_1^{[n]}), \mathrm{pr}_a^*(w_2^{[n]}), w_1^{\mathrm{diff}}, w_2^{\mathrm{diff}}, \mathrm{pr}_a^*(D_n), \mathrm{pr}_b^*(D_{n+1}) \}.
    \]
\end{corollary}

In fact, using the complete description (see \cite[Theorem 12.1]{Bayer-Macri-K3}) of the nef cone of $X^{[n]}$ for any $n$, where $X$ is a K3 surface, one can similarly compute the nef cone of $X^{[n,n+1]}$ for any $n$.

\subsection*{Organization of the paper}
In \S\ref{sec:Prelims}, we review the necessary background on K3 surfaces, Hilbert schemes, and Bridgeland stability conditions, and outline our overall strategy (see \S\ref{strategy}). In \S\ref{sec:NefCones}, we present the main computations. \S\ref{subsec:large_n} focuses on the large $n$ case, where we use techniques from \cite{Bol+:2016} to determine nef cones. Finally in \S\ref{subsec:small_n} and \S\ref{subsec:nef_cone_nk2_lattice_computation} we address the small $n$ case and directly apply \cite{Bayer-Macri-K3}.

\subsection*{Acknowledgments} We would like to thank Izzet Coskun, Jack Huizenga, and Howard Nuer for several helpful discussions. We also express our gratitude toward Arend Bayer and Emanuele Macr\`{i} for their comments on an earlier draft of this article that drew our attention to the results in \cite{Bayer-Macri-K3}, and made the exposition clearer and more historically accurate.

\section{Preliminaries}\label{sec:Prelims}

In this section, we remind the reader of some definitions and collect some basic facts about K3 surfaces, Hilbert schemes of points, and Bridgeland stability conditions, along with their usage in determining nef cones.

\subsection{K3 surfaces}
Throughout this paper, we shall work over an algebraically closed field $k$ of characteristic zero. 
\begin{definition}
    A \emph{K3 surface} over $k$ is a complete non-singular variety X of dimension
two such that
\begin{align*}
    \omega_X\cong\mathcal{O}_X\text{ and }H^1(X,\mathcal{O}_X)=0.
\end{align*}
\end{definition}
We now turn our interest to Mori dream K3 surfaces of Picard rank $\rho(X)=2.$ Since the specific properties of Mori dream spaces are not used in this paper, we refer to \cite[Ch. 5]{Coxrings:2015} for definitions and background.

By \cite[Th. 5.1.5.3]{Coxrings:2015}, a Picard rank $2$ K3 surface is a Mori dream surface if and only if there exists $w$ in the divisor class group $\operatorname{Cl}(X)$, such that $w^2 \in\{0,-2\}.$ The following result is the starting point of this paper.

\begin{proposition}{\cite[Proposition 5.3.2.4]{Coxrings:2015}}\label{Prop. 2.3}
     Let $X$ be a K3 surface with $\operatorname{Cl}(X)\cong \mathbb{Z}w_1\oplus\mathbb{Z}w_2$ for some smooth curves $w_1$ and $w_2$ on $X$. Then the following are equivalent.
    \begin{enumerate}
        \item The intersection matrix of $w_1,w_2$ is one of the following, where $k\geq 2$ in case $I$ and $k\geq 1$ in the remaining cases:
        \begin{align*}
        \mbox{(I)}:\begin{bmatrix}
            0 & k\\
            k & 0
        \end{bmatrix}\hspace{1cm} \mbox{(II)}:\begin{bmatrix}
            -2 & k\\
            k & 0
        \end{bmatrix}\hspace{1cm} \mbox{(III)}:\begin{bmatrix}
            -2 & k\\
            k & -2
        \end{bmatrix}
    \end{align*}
        \item The effective cone $\operatorname{Eff}(X)$ is polyhedral and, after possibly substituting $w_i$   with $-w_i$, it is generated as a cone by $w_1, w_2$.
    \end{enumerate}
\end{proposition}

Since $\operatorname{Nef}(X)$ is the dual of the effective cone, we list the nef cones in each case of the above proposition.
\begin{enumerate}
    \item $\operatorname{Nef}(X)=\operatorname{Cone}(w_1,w_2).$
    \item $\operatorname{Nef}(X)=\operatorname{Cone}(kw_1+2w_2,w_2).$
    \item $\operatorname{Nef}(X)=\operatorname{Cone}(kw_1+2w_2,2w_1+kw_2).$
\end{enumerate}
The adjunction formula gives us
\begin{equation} \label{genusformula}
    g(C)=\frac{C^2+2}{2}.
\end{equation}
Hence, 
\begin{enumerate}
    \item $g(w_1)=1=g(w_2).$
    \item $g(w_1)=0,\ g(w_2)=1.$
    \item $g(w_1)=0=g(w_2).$
\end{enumerate}

\subsection{Hilbert schemes of points}\label{Hilbscheme}
Let $X$ be a smooth projective surface with $h^1(X,\mathcal{O}_X)=0.$ By the work of Fogarty (see {\cite{Fogarty}}), we know the Hilbert scheme of $n$ points $X^{[n]}$ is a smooth projective variety of dimension $2n$ which resolves the singularities of the symmetric product $X^{(n)}$ via the Hilbert-Chow morphism $$X^{[n]}\to X^{(n)}.$$ 

A line bundle $L$ on $X$ induces a line bundle $L^{(n)} 
\coloneq \bigotimes_i pr_i^*L$ on the symmetric product, and by pulling back via the Hilbert-Chow morphism one defines a line bundle on $X^{[n]},$ which we will denote by $L^{[n]}$. \cite{Fogarty} shows that
$$\operatorname{Pic}(X^{[n]})\cong \operatorname{Pic}(X)\oplus \mathbb{Z}(B^{[n]}/2),$$
where $\operatorname{Pic}(X)\hookrightarrow \operatorname{Pic}(X^{[n]})$ by $L\mapsto L^{[n]},$ and $B^{[n]}$ is the locus of nonreduced schemes, i.e. the exceptional divisor of the Hilbert–Chow morphism. We will often drop the superscript on $B^{[n]}$ when it is unlikely to result in confusion. Tensoring by the real numbers, the Néron–Severi space $N^1(X^{[n]})$ is therefore spanned by $N^1(X)$ and $B.$

We denote by $C_0$ the class of a curve contracted
by the Hilbert-Chow morphism. There are many ways curves on $X$ induce curve classes on $X^{[n]}.$ In this paper, we will use the following construction. Let $C\subset X$ be a reduced and irreducible curve. If $C$ admits a degree $n$ map to $\mathbb{P}^1$ (i.e. a $g^1_n$), then the fibers of $C\to\mathbb{P}^1$ give a rational curve $\mathbb{P}^1\to X^{[n]}.$ We write $C_{[n]}$ for this class.

Let $C$ be a curve and $D$ be a divisor on $X$. The following intersection numbers are standard (see \cite{Fogarty}):
\[
\begin{array}{c|cc}
 & D^{[n]} & B \\[4pt]
 \hline\\[-9pt]
C_{[n]} & C\cdot D & 2g(C)-2+2n \\[4pt]
C_0 & 0 & -2
\end{array}
\]

\subsection{Nested Hilbert schemes of points}

A natural generalization of the above Hilbert schemes is the notion $\NestHilb$ of \emph{nested Hilbert schemes}. The scheme $\NestHilb$ parametrizes $(z,z')\in X^{[n]}\times X^{[n+1]}$ such that $z$ is a subscheme of $z'$. By \cite[Chapter 1]{Cheah:1998}, $X^{[n,n+1]}$ is a smooth projective variety of dimension $2n+2$.

We introduce some notation. Let $\prj_a:\NestHilb\to X^{[n]}$ and $\prj_b:\NestHilb\to X^{[n+1]}$ be the projections. Given a line bundle $L$ on $X$, let $L^a:=\prj_a^*(L^{[n]})$, $L^b:=\prj_b^*(L^{[n+1]})$. Also let $B^a:=\prj_a^*(B^{[n]})$ and $B^b:=\prj_b^*(B^{[n+1]})$. Moreover, consider the \emph{residuation} morphism $\res:\NestHilb\to X$, which is the map that takes a tuple $(z,z')$ as above to the scheme $z'\backslash z$. We then have the following divisors $$B^{\diff}:=B^b-B^a=\{\xi=(z,z')\in\NestHilb\,:\,\Supp(z)\cap\res(\xi)\neq\emptyset\}$$ and $L^{\diff}:=L^b-L^a$. If $L$ is an effective line bundle on $X$ with a general representative $D\in|L|$, then $$L^{\diff}=\{\xi=(z,z')\in\NestHilb\,:\,D\cap\res(\xi)\neq\emptyset\}.$$

By \cite[Proposition 3.1]{Ryan-Yang:2020}, $$\Pic(X^{[n,n+1]})=\prj_a^*(X^{[n]})\oplus\prj_b^*(X^{[n+1]})=\Pic(X)^2\oplus\mathbb{Z}\!\cdot\!\frac{B_a}{2}\oplus\mathbb{Z}\!\cdot\!\frac{B_b}{2}.$$

We also recall some standard curves on the nested Hilbert scheme $\NestHilb$. \begin{enumerate}
    \item For $i=1,2$, let $C_i^a$ be the curve on $\NestHilb$ given by fixing $n$ general points of $X$ as part of $z'$ (including the residual point), and varying a point on a curve of class $w_i$ (and including this point in $z$).
    \item For $i=1,2$, let $C_i^b$ be the curve given by fixing $n$ general points of $X$ as $z$, and varying a point on a curve of class $w_i$ to get $z'$.
    \item Let $C_0^a$ be the curve given by fixing $n-2$ general points of $X$ as part of $z$, fixing one general point on a curve of class $w_1$ (and including it in $z$), fixing a general point of $X$ as the residual point, and varying one point on the same curve of class $w_1$ as above (and including it in $z$).
    \item Let $C_0^b$ be the curve given by fixing $n-1$ general points of $X$ as part of $z$, fixing one general point on a curve of class $w_1$ (and including it in $z$), and varying the residual point on the same curve of class $w_1$.
    \item Let $A^a$ be the curve obtained by fixing $n-1$ general points of $X$ with a scheme of length $3$ supported at the residual point (all of these constitute $z')$, with $z$ being the subscheme of $z'$ obtained by including all the reduced points, and varying a double-point structure at the nonreduced point.
    \item Finally, let $A^b$ be the curve obtained by fixing $n$ points of $X$ as $z$, and varying a double-point structure at one of those points (fixed) to get the residual.
\end{enumerate} We refer the reader to \cite[\S4.2]{Ryan-Yang:2020} for pictures explaining these curves (although we must caution the reader that their projections are reversed compared to ours, whence their superscripts $a$ and $b$ are reversed as well).

Let $C_1, C_2$ be curves of classes $w_1,w_2$ respectively. We note the following intersection products from \cite[4.3.3]{Ryan-Yang:2020} (Most of these computations follow from the push-pull formula): \[ \label{nestedtable}
\begin{array}{c|ccccccccc}
 & w_1^a & w_1^b & w_1^{\diff} & w_2^a & w_2^b & w_2^{\diff} & B^a & B^b & B^{\diff} \\[4pt]
 \hline\\[-9pt]
C_1^a & w_1 \cdot C_1 & w_1 \cdot C_1 & 0 & w_2 \cdot C_1 & w_2 \cdot C_1 & 0 & 0 & 0 & 0 \\[4pt]
C_2^a & w_1 \cdot C_2 & w_1 \cdot C_2 & 0 & w_2 \cdot C_2 & w_2 \cdot C_2 & 0 & 0 & 0 & 0 \\[4pt]
C_1^b & 0 & w_1 \cdot C_1 & w_1 \cdot C_1 & 0 & w_2 \cdot C_1 & w_2 \cdot C_1 & 0 & 0 & 0 \\[4pt]
C_2^b & 0 & w_1 \cdot C_2 & w_1 \cdot C_2 & 0 & w_2 \cdot C_2 & w_2 \cdot C_2 & 0 & 0 & 0 \\[4pt]
C_0^a & w_1 \cdot C_1 & w_1 \cdot C_1 & 0 & w_2 \cdot C_1 & w_2 \cdot C_1 & 0 & 2 & 2 & 0 \\[4pt]
C_0^b & 0 & w_1 \cdot C_1 & w_1 \cdot C_1 & 0 & w_2 \cdot C_1 & 0 & 0 & 2 & 2 \\[4pt]
A^a & 0 & 0 & 0 & 0 & 0 & 0 & -2 & 0 & 2 \\[4pt]
A^b & 0 & 0 & 0 & 0 & 0 & 0 & 0 & -2 & -2
\end{array}
\]

\subsection{Bridgeland stability conditions}
In this section, we recall some basic definitions and properties of Bridgeland stability conditions. For more details, we refer to \cite{BriOG:2007,BriK3:2008,Macri-Schmidt:2017}. 

A Bridgeland stability condition on $X$ is a pair $\sigma=(Z,\mathcal{A}),$ where $\mathcal{A}\subset D^b(X)$ is the heart of a bounded $t$-structure in the derived category of coherent sheaves of $X,$ and $Z:K_0(\mathcal{A})\to\mathbb{C}$ is an additive homomorphism that satisfies for every $E\in K_0(\mathcal{A})$,
\begin{itemize}
    \item $\operatorname{Im}(Z(E))\geq 0,$
    \item $\operatorname{Im}(Z(E))=0\implies \operatorname{Re}(Z(E))<0.$
\end{itemize}
The technical parts of Bridgeland stability conditions are the existence of the Harder-Narasimhan filtration and support property (see \cite[Definition 5.2]{Macri-Schmidt:2017}). In the case of surfaces, Bridgeland (see \cite{BriK3:2008}) and Arcara and Bertram (see \cite{Arcara-Bertram:2013}) showed how to construct Bridgeland stability conditions in a slice corresponding to a choice
of an ample divisor $H\in\operatorname{Pic}(X)_{\mathbb{R}}$ and an arbitrary divisor $D\in\operatorname{Pic}(X)_{\mathbb{R}}$. The twisted Mumford slope function is defined as follows,
\begin{align*}
    \mu_{H,D}(E)=\frac{H\cdot\operatorname{ch}_1^D(E)}{H^2\cdot\operatorname{ch}_0^D(E)}.
\end{align*}
The twisted Chern characters are defined as follows
\begin{align*}
    \operatorname{ch}_0^D(E)= \operatorname{ch}_0(E), \ \operatorname{ch}_1^D(E)=\operatorname{ch}_1(E)-\operatorname{ch}_0(E)D, \ \operatorname{ch}_2^D(E)=\operatorname{ch}_2(E)-\operatorname{ch}_1(E)\cdot D+\frac{1}{2}\operatorname{ch}_0(E)D^2.
\end{align*}
For  a given real number $s\in\mathbb{R},$ there is a torsion pair $(\mathcal{T}_{s},\mathcal{F}_{s})$ in $\operatorname{Coh}(X)$, where 
\begin{align*}
    \mathcal{T}_{s}&=\{T\in \operatorname{Coh}(X)\ |\ T\text{ is torsion or }\mu_{H,D}(T')> s\text{ for all quotients $T'$ of $T$}\}\\
    \mathcal{F}_{s}&=\{F\in \operatorname{Coh}(X)\ |\ F\text{ is torsion-free and }\mu_{H,D}(F')\leq s\text{ for all proper subsheaves $F'$ of $F$}\}
\end{align*}
Tilting gives a heart $\mathcal{A}_s=\langle \mathcal{F}_{s}[1],\mathcal{T}_s\rangle,$ which is the smallest extension closed full subcategory of $D^b(X)$ containing both $\mathcal{T}_s$ and $\mathcal{F}_s[1].$ For an additional $t\in\mathbb{R}_{>0},$ define the homomorphism
\begin{align*}
    Z_{s,t}(E)=-\operatorname{ch}_2^{D+sH}(E)+\frac{t^2H^2}{2}\operatorname{ch}_0^{D+sH}(E)+iH\cdot \operatorname{ch}_1^{D+sH}(E).
\end{align*}
The $\mu_{s,t}$-slope of an object $E\in\mathcal{A}_s$ is defined by
$$\mu_{s,t}(E)=-\frac{\operatorname{Re}Z_{s,t}(E)}{\operatorname{Im}Z_{s,t}(E)}.$$
Then $\sigma_{s,t}=(Z_{s,t},\mathcal{A}_s)$ is a Bridgeland stability condition on $X.$ The $(H,D)$-slice of stability conditions is the family of stability conditions
\begin{align*}
    \{\sigma_{s,t}\ |\ s,t\in\mathbb{R},t>0\}
\end{align*}
parametrized by $(s,t)$ upper half plane. This upper half plane admits a wall-and-chamber decomposition such that stable objects of Chern character $v$ cannot destabilize unless a wall is crossed. Maciocia\cite{Maciocia:2014} showed that these walls are either the vertical line $s=\mu_{H,D}(v)$ or nested semicircles. If $W_1$ and $W_2$ are two semicircular numerical walls left of the vertical wall with centers $(s_{W_1},0)$ and $(s_{W_2},0)$, then $W_2$ is nested inside $W_1$ if and only if $s_{W_1}<s_{W_2}.$

Given objects $E$ and $F$ in $\mathcal{A}_s$, the numerical wall $W(E,F)$ in the $(H, D)$-slice consisting of stability conditions for which $E$ and $F$ have the same $\mu_{s,t}$-slope has center $s_0$ and radius $\rho$ satisfying
\begin{align*}
    s_0&=\frac{1}{2}\left(\mu_{H,D}(E)+\mu_{H,D}(F)\right)-\frac{\Delta_{H,D}(E)-\Delta_{H,D}(F)}{\mu_{H,D}(E)-\mu_{H,D}(F)}\\
    \rho^2&=\left(\mu_{H,D}(E)-s_0\right)^2-2\Delta_{H,D}(E).
\end{align*}

Here $\Delta_{H,D}(E)=\frac{1}{2}\mu_{H,D}(E)^2-\frac{\operatorname{ch}_2^D(E)}{H^2\operatorname{ch}_0^D(E)}.$ In particular, if $I_Z$ is the ideal sheaf of a length n scheme and $E=L\otimes I_{Z'}$ for
some line bundle $L$ and $0$-dimensional scheme $Z'$ of length $m$, then the center $s_0$ of $W(E,I_Z)$ is given by
\begin{align*}
    s_0=\frac{n-m+\frac{L^2}{2}-L\cdot D}{L\cdot H}.
\end{align*}

\begin{lemma}{\cite{BriK3:2008}}
    (Large volume limit) Fix divisors $(H,D)$ giving a slice in $\operatorname{Stab}(X)$. If $E\in\mathcal{A}_s$ has $\operatorname{ch}(E)=v$ (with $\operatorname{ch}_0(v)>0$), then $E$ is $\sigma_{s,t}$-stable for $t\gg 0$ if and only if $E$ is a $(H,D-K_X/2)$-twisted Gieseker semistable sheaf.
    
    Moreover, in the quadrant of the $(H, D)$-slice left of the vertical wall there is a
    largest semicircular wall for $v$, called the Gieseker wall. For all $(s,t)$ between this wall and the vertical wall, the Bridgeland moduli space $M_{\sigma_{s,t}}(v)$ coincides with the moduli space $M_{H,D-K_X/2}(v)$ of $(H,D-K_X/2)$-twisted Gieseker semistable sheaves.
\end{lemma}
Since $v=(1,0,-n)$ is the Chern character for an ideal sheaf $I_Z\in X^{[n]}$ of $n$ points on $X$, following the above lemma, for a stability condition $\sigma$ lying above the Gieseker wall, we have $M_{\sigma}(v)\cong X^{[n]}.$ Let us briefly explain our strategy.\\

\subsection{Strategy:}\label{strategy} Following the notation in \S\ref{Hilbscheme}, suppose $\Lambda\subset N^1(X^{[n]})$ is the cone of all divisors on $X^{[n]}$ that have non-negative intersections with $w_{1,[n]},w_{2,[n]}$ and $C_0$. Clearly, $\operatorname{Nef}(X^{[n]})$ is contained in $\Lambda$. As usual, we conflate $\operatorname{Nef}(X)$ with its image in $\operatorname{Nef}(X^{[n]})\subset N^1(X^{[n]})$ under the map $L\mapsto L^{[n]}$.

We then show that $\Lambda$ is the cone in $N^1(X^{[n]})$ spanned by $\operatorname{Nef}(X^{[n]})$ and another divisor $D.$ The final step is to show that $D$ is a nef divisor using the following result.
\begin{proposition}{\cite[Proposition 3.8]{Bol+:2016}}\label{main Bol+ theorem}
Let $X$ be a smooth projective surface with irregularity
zero, $H$ an ample divisor, and $D$ an antieffective $\mathbb{Q}$-divisor. Let $\sigma$ be a stability condition lying on a numerical wall with center $s_W$ in the Gieseker chamber in the $(H,D)$-slice corresponding
to the Chern character $v=(1,0,-n)$. Then the divisor
$$\frac{1}{2}K_X^{[n]}-s_WH^{[n]}-D^{[n]}-\frac{1}{2}B$$
is nef.
\end{proposition}

\section{The nef cones}\label{sec:NefCones}
In this section, we shall compute the nef cones of Hilbert schemes of points in all three cases of Proposition \ref{Prop. 2.3}. 

\subsection{The nef cone for large $n$}\label{subsec:large_n}

We begin with case I. The first goal is to describe $\Lambda$ (defined in subsection \ref{strategy}) more explicitly.

\begin{lemma}\label{lem:characterize_Lambda}
	Suppose $n>1$ is an integer. Then the cone $\Lambda$ is the cone in $N^1(X^{[n]})$ spanned by $\operatorname{Nef}(X)$ and the class $$\frac{n}{k}w_1^{[n]}+\frac{n}{k}w_2^{[n]}-\frac{1}{2}B.$$
\end{lemma}

\begin{proof}
	Let $D=M^{[n]}-xB\in \Lambda$ be arbitrary, with $M\in N^1(X)$ and $x\in\mathbb{R}$. Since $0\leq D\cdot C_0=2x$, we have $x\geq 0.$
	
	From the definition of $\Lambda,$ we get $$0\leq D\cdot w_{i,[n]}=M\cdot w_i-x(2g(w_i)+2n-2)=M\cdot w_i-2nx,$$ which yields \begin{equation}\label{eq:M_is_very_nef}
	    M\cdot w_i\geq 2nx\geq 0
	\end{equation} for $i=1,2.$
	
	Hence, any $D\in \Lambda$ is of the form $$D=M^{[n]}-xB$$ with $x\geq 0$ and $M\in\operatorname{Nef}(X)$.
	Scaling by $\mathbb{Q}_{>0}$ we now assume that if $D\in\Lambda\backslash\operatorname{Nef}(X)$, then 
    $$D=M^{[n]}-\frac{1}{2}B$$
    with $M\in\operatorname{Nef}(X)$. Write
    $$D=\left(M^{[n]}-\alpha w_1^{[n]}-\alpha w_2^{[n]}\right)+\left(\alpha w_1^{[n]}+\alpha w_2^{[n]}-\frac{1}{2}B\right)$$
    where $\alpha:=n/k$. Notice 
    $$(M-\alpha w_1-\alpha w_2)\cdot w_1=M\cdot w_1-\alpha k=M\cdot w_1-n\geq 0$$
    by inequality \eqref{eq:M_is_very_nef} above. Similarly, $$(M-\alpha w_1-\alpha w_2)\cdot w_2\geq0.$$ 
    Since a divisor on $X$ is nef if and only if it has non-negative intersections with $w_1$ and $w_2$, it follows that $M-\alpha w_1-\alpha w_2$ is nef on $X$, from which we see that $M^{[n]}-\alpha w_1^{[n]}-\alpha w_2^{[n]}$ is nef on $X^{[n]}$. Thus, $\Lambda$ is contained in the cone in $N^1(X^{[n]})$ spanned by $\operatorname{Nef}(X)$ and $$\alpha w_1^{[n]}+\alpha w_2^{[n]}-\frac{1}{2}B.$$

    On the other hand, we note that $$\alpha w_1^{[n]}+\alpha w_2^{[n]}-\frac{1}{2}B$$ meets each of $w_{1,[n]},w_{2,[n]},$ and $C_0$ non-negatively, and therefore belongs to the cone $\Lambda$.
\end{proof}

As explained in the strategy, we shall now prove that $\frac{n}{k}w_1^{[n]}+\frac{n}{k}w_2^{[n]}-\frac{1}{2}B$ itself is a nef divisor on $X^{[n]}$ for large $n.$

\begin{theorem}\label{thm:Nef_cone_case1}
	For any integer $n\geq9k/8$, the nef cone of the Hilbert scheme $X^{[n]}$ is the cone $\Lambda$ constructed above. Dually, the Mori cone $\overline{\operatorname{NE}}(X^{[n]})$ is the cone in $N_1(X^{[n]})$ spanned by the curve classes $w_{1,[n]},w_{2,[n]}$ and $C_0$.
\end{theorem}

\begin{proof}
	Let $n>k$ be any integer. Let $H$ be the ample divisor class $$\left(\frac{n}{k}-1\right)w_1+\left(\frac{n}{k}-1\right)w_2$$ on $X$, and let $D$ be the anti-effective divisor class $-w_1-w_2$. The corresponding set of critical effective divisors (as defined in \cite[\S3, page 917]{Bol+:2016}) is seen to be $\operatorname{CrDiv}(H,D)=\left\{w_1+w_2,w_1,w_2\right\}$. Calculating the centers of the numerical walls for the Chern character $\mathbf{v}(I_Z)$ (where $Z$ is a length $n$ subscheme of $X$) in the $(H,D)$-slice given by $\mathcal{O}_X(-F)$ for each of the three critical effective divisors $F$, we get $$s_{W(\mathcal{O}(-w_1-w_2),I_Z)}=-1/2 \text{ and } s_{W(\mathcal{O}(-w_1),I_Z)}=s_{W(\mathcal{O}(-w_2),I_Z)}=-1,$$
    where $(s_W,0)$ denotes the center of the numerical wall $W$. Since the numerical walls for $\mathbf{v}(I_Z)$ are nested semicircles, it follows that $W(\mathcal{O}_X(-w_1),I_Z)$ is the largest numerical wall given by any critical effective divisor. We calculate the radius $\rho$ of the wall $W(\mathcal{O}_X(-w_1),I_Z)$ and find that it satisfies $$\rho^2=(\mu_{H,D}(I_Z)-(-1))^2-2\Delta_{H,D}(I_Z)=\frac{n}{n-k}.$$ 
    
    We see that $\rho^2\geq\varrho$ for $n\geq9k/8$, where $$\varrho:=\frac{2nH^2+(H\cdot D)^2-H^2D^2}{8(H^2)^2}=\frac{nk}{8(n-k)^2}.$$ By \cite[Theorem 3.6]{Bol+:2016}, $W(\mathcal{O}_X(-w_1),I_Z)$ is the Gieseker wall for $\mathbf{v}(I_Z)$ in the $(H,D)$-slice.
	
	Now, Theorem \ref{main Bol+ theorem} show that the divisor class $$\frac{1}{2}K_X^{[n]}-(-1)A^{[n]}-P^{[n]}-\frac{1}{2}B^{[n]}=\frac{n}{k}w_1^{[n]}+\frac{n}{k}w_2^{[n]}-\frac{1}{2}B^{[n]}$$ is nef on $X^{[n]}$. By Lemma \ref{lem:characterize_Lambda}, $\Lambda\subseteq\operatorname{Nef}(X^{[n]})$. Since the other inclusion $\operatorname{Nef}(X^{[n]})\subseteq\Lambda$ is obvious, the assertion follows.
\end{proof}

Similarly, we also find the nef cones for cases II and III (Proposition \ref{Prop. 2.3}).

\begin{theorem}\label{thm:Nef_cone_case2and3}
    \begin{enumerate}
        \item Case II: For any large enough integer $n$, the nef cone of the Hilbert scheme $X^{[n]}$ is the cone $\Lambda$ constructed above. Dually, the Mori cone $\overline{\operatorname{NE}}(X^{[n]})$ is the cone in $N_1(X^{[n]})$ spanned by the curve classes $w_{1,[n]},w_{2,[n]}$ and $C_0$.

        More precisely, the above assertions hold whenever $n$ satisfies the inequalities $n>k$ and $$32(k^2+2k+1)n^4-4k(25k^2+41k+16)n^3+4k^2(26k^2+33k+7)n^2-4k^3(9k^2+8k-1)n-k^4\geq0.$$

    \item Case III: Assume $k\geq 3$. For any large enough integer $n$, the nef cone of the Hilbert scheme $X^{[n]}$ is the cone $\Lambda$ constructed above. Dually, the Mori cone $\overline{\operatorname{NE}}(X^{[n]})$ is the cone in $N_1(X^{[n]})$ spanned by the curve classes $w_{1,[n]},w_{2,[n]}$ and $C_0$.

    More precisely, the above assertions hold whenever $n$ satisfies the inequalities $n\geq k\geq 3$ and $$8n^2+(2-9k)n+8\geq0.$$
    \end{enumerate}
\end{theorem}

\subsection{Bounds for small $n$}\label{subsec:small_n}

When $n$ is small, we use more classical geometric techniques to find bounds on the nef cones of the Hilbert schemes of points on K3 surfaces. For simplicity, we assume the following.

\begin{notation} \label{notation}
For the remainder of this article, we restrict to the case that $X$ is a K3 surface of type I (see Proposition \ref{Prop. 2.3}).
\end{notation}

If $a_1w_1^{[n+1]}+a_2w_2^{[n+1]}-\frac{B^{[n+1]}}{2}$ is nef on $X^{[n+1]}$, then $a_1w_1^{[n]}+a_2w_2^{[n]}-\frac{B^{[n]}}{2}$ is nef on $X^{[n]}$. This is because of the following reason: suppose $a_1w_1^{[n]}+a_2w_2^{[n]}-B^{[n]}/2$ is negative on a curve $C$ on $X^{[n]}$. Then the same curve $C$ can be viewed as a curve on $X^{[n+1]}$ by ``adding'' a fixed general point of $X$ to each scheme on $C$ and this does not change intersection numbers. So, $\Nef(X^{[n]})\supseteq\Nef(X^{[n+1]})$, where we identify $w_i^{[n]}$ with $w_i^{[n+1]}$ for $i=1,2$, and identify $B^{[n]}$ with $B^{[n+1]}$. We are going to use this fact in the proof below.

As mentioned above, we are dealing with case I of K3 surfaces. Suppose $n'\geq9k/8$ and $n'-1<9k/8$, so that the Bridgeland stability method works for $n'$ but not for $n'-1$. We then have the following result.

\begin{proposition}\label{prop:bound_for_small_n_case1}
    Let $X$ be as in \ref{notation}, and let $n'$ be as defined just above. Let $n$ be a positive integer with $1<n<n'$, and $j:=n'-n$. Then the nef cone of $X^{[n]}$ is sandwiched between the cones $\Lambda$ and $\Lambda'$, where $\Lambda$ is the cone in $N^1(X^{[n]})$ spanned by $\Nef(X)$ and $$\frac{n}{k}w_1^{[n]}+\frac{n}{k}w_2^{[n]}-\frac{1}{2}B^{[n]},$$ and $\Lambda'$ is the cone in $N^1(X^{[n]})$ spanned by $\Nef(X)$ and $$\frac{n+j}{k}w_1^{[n]}+\frac{n+j}{k}w_2^{[n]}-\frac{1}{2}B^{[n]}.$$
\end{proposition}

\begin{proof}
    By Theorem \ref{thm:Nef_cone_case1}, we know that the nef cone of $X^{[n']}$ is spanned by $\Nef(X)$ and the class $$\frac{n'}{k}w_1^{[n']}+\frac{n'}{k}w_2^{[n']}-\frac{1}{2}B^{[n']}.$$ So $$\frac{n+j}{k}w_1^{[n]}+\frac{n+j}{k}w_2^{[n]}-\frac{1}{2}B^{[n]}$$ is nef on $X^{[n]}$. On the other hand, by Lemma \ref{lem:characterize_Lambda}, the nef cone of $X^{[n]}$ is contained in the cone $\Lambda$.
\end{proof}

\begin{remark}
    The smaller $j/k$ is, the better our bound. Note that since $n'<\dfrac{9k}{8}+1$ and $n>1$, we have $0<j<9k/8$, i.e., $0<j/k<9/8=1.125$.
\end{remark}

We can improve the above bound when $k$ is large enough compared to $n$. We still consider case I of K3 surfaces, and maintain the same notation as just above. Suppose $k\geq 2n$. Let $L$ be the ample divisor $w_1+w_2$ on $X$. We claim that it is $n$-very ample. Following \cite[Theorem 1.1]{Knutsen:2001}, it suffices to show that $L^2\geq4n$, and that there is no effective divisor $D$ on $X$ satisfying the following conditions: \begin{enumerate}
	\item $2D^2 \stackrel{\mathclap{\mbox{(i)}}}{\leq} L\cdot D \leq D^2+n+1 \stackrel{\mathclap{\mbox{(ii)}}}{\leq} 2n+2$;
	\item Inequality (i) is an equality if and only if $L\sim 2D$ and $L^2\leq 4n+4$;
	\item Inequality (ii) is an equality if and only if $L\sim 2D$ and $L^2=4n+4$.
\end{enumerate} 
Since $L^2=2k\geq 4n$, we only need to show the effective divisor condition. Suppose $D=xw_1+yw_2$ is an effective divisor on $X$ satisfying (1). Then $2kxy+n+1 = D^2+n+1\leq 2n+2$, which implies that $2kxy\leq n+1$. But $k\geq2n$, which means that the left side is at least $4nxy$, whence $(4xy-1)n\leq 1$. Since $x,y,n$ are non-negative integers (with $n>0$ of course), it follows that $xy=0$. Without loss of generality, say $x=0$. Then $$ky=L\cdot D\leq D^2+n+1=2kxy+n+1=n+1.$$ Again, $ky\geq 2ny$, which implies that $(2y-1)n\leq 1$. Similar reasoning as above yields $y=0$, which means that $D=0$. So inequality (1) is an equality. However, $L\not\sim2D=0$. This proves that $L$ is $n$-very ample. By \cite[Main Theorem]{Catanese-Goettsche:1990}, $L$ induces an embedding $\varphi:X^{[n]}\rightarrow G$, where $G$ denotes the Grassmannian 
$\operatorname{Gr}\!\left(n,H^0(X,L)\right)$. As mentioned in \cite[\S2]{Bertram-Coskun:2013}, one can use the Grothendieck-Riemann-Roch Theorem to show that the class of $\varphi^*(\sigma_1)$ in the N\'{e}ron-Severi space of $X^{[n]}$ is $$w_1^{[n]}+w_2^{[n]}-\frac{1}{2}B^{[n]}.$$ Since $\varphi$ is an embedding, this class is ample on $X^{[n]}$. Note that $$\frac{n+j}{k}=\frac{n'}{k}\geq\frac{9}{8}>1,$$ which means we now have a better bound: The nef cone of $X^{[n]}$ is sandwiched between the cone in $N^1(X^{[n]})$ spanned by $\Nef(X)$ and $$\frac{n}{k}w_1^{[n]}+\frac{n}{k}w_2^{[n]}-\frac{1}{2}B^{[n]},$$ and the cone spanned by $\Nef(X)$ and $$w_1^{[n]}+w_2^{[n]}-\frac{1}{2}B^{[n]}.$$

We record this below.

\begin{proposition} \label{bounds}
    Let $X$ be a as in \ref{notation}, and let $k\geq 2n$ be as defined just above. The nef cone of $X^{[n]}$ is sandwiched between the cone in $N^1(X^{[n]})$ spanned by $\Nef(X)$ and $$\frac{n}{k}w_1^{[n]}+\frac{n}{k}w_2^{[n]}-\frac{1}{2}B^{[n]},$$ and the cone spanned by $\Nef(X)$ and $$w_1^{[n]}+w_2^{[n]}-\frac{1}{2}B^{[n]}.$$
\end{proposition}

On the other hand, in the case $n=k$, we can find curves orthogonal to $$w_1^{[n]}+w_2^{[n]}-\frac{1}{2}B^{[n]}$$ on $X^{[n]}$. Let $C$ be a smooth member of the linear series $|w_1+w_2|$. Assuming $C$ carries a $g^1_k$, we get $$\left(w_1^{[k]}+w_2^{[k]}-\frac{1}{2}B^{[k]}\right)\cdot C_{[k]}=2k-\frac{1}{2}(2g(w_1)-2+2k+2g(w_2)-2+2k)=2k-2k=0.$$

\begin{remark}
    The following shows that even if we did not bound $\Nef(X^{[n]})$ by two closed cones (and later show that it coincides with one of them), we can directly (in this specific case) explicitly show that the desired ray is extremal.
\end{remark}

It is enough to show that $C$ carries a $g^1_k$. Suppose $n=k\geq3$. We recall that the Brill-Noether Theorem tells us, in particular, that if the Brill-Noether number $$\rho_{g,r,d}:=g-(r+1)(g-d+r)$$ is non-negative, then any smooth curve of genus $g$ carries a $g^r_d$. Now in our case, we compute the genus of $C$ to be $k+1$ using the adjunction formula in the form of \ref{genusformula}, so that $$\rho_{k+1,1,k}=k+1-2(k+1-k+1)=k-3\geq0.$$ Thus, $C$ carries a $g^1_k$.

For the case $n=k=2$, we have a more geometric argument. First, we observe that the line bundle $w_1 \otimes w_2$ on $X$ is ample, but not very ample: Let $\Gamma$ be a smooth member of $|w_1|$. Then since $w_1 \otimes w_2$ is ample, $(w_1 \otimes w_2)|_{\Gamma}$ is also ample. As $\Gamma$ is a curve of genus $1$, the restriction $(w_1 \otimes w_2)|_{\Gamma}$ must be a line bundle of degree 2 on a smooth genus 1 curve. Such a line bundle has two sections, so it cannot give an embedding $\Gamma\hookrightarrow \mathbb{P}^n$ for any $n$. So $w_1 \otimes w_2$ is not very ample. Now, let $\varphi$ be the map corresponding to the complete linear series $|w_1 + w_2|$ on $X$. By Riemann-Roch and Kodaira Vanishing, $w_1 + w_2$ has four sections. Hence $\mathrm{im}(\varphi)=Q\subset\mathbb{P}^3$, where $Q$ is a smooth quadric surface. The images of $w_1$ and $w_2$ under this map are (respectively) the rulings $L_1$ and $L_2$ of $Q$. Note that $g(C)=3$, where $C$ is a general (in particular, smooth) element of the linear series $|w_1+w_2|$. Consider a smooth curve of class $L_1+L_2$ on $Q$. This curve is rational, and since $\varphi$ is $2$-to-$1$, we obtain a $g^1_2$ on $C$.

Hence, we have the following result:

\begin{proposition}
    Let $X$ be as in \ref{notation}. For the case $n=k$, there exists an effective class $[C]$ in $N_1(X^{[n]})$ dual to the class $\Gamma \coloneq w_1^{[n]} + w_2^{[n]} - \frac{1}{2}B^{[n]}$.
\end{proposition}

Finally, we want to show that when $n = k = 2$, the divisor for which we produced a dual curve above is actually nef. At that point, $\Nef(X^{[2]})$ must be exactly the cone $\Lambda$ in \ref{bounds}, so since $\Lambda$ was a bound, the ray generated by this divisor must be extremal. The existence of the dual curve also guarantees that it is extremal. Thus we get the following theorem:

\begin{theorem}\label{thm:Nef_cone_nk2}
    In case I when $k = 2$, the nef cone of the Hilbert scheme $X^{[2]}$ is spanned by $w_1^{[2]}, w_2^{[2]}$ and $$w_1^{[2]} + w_2^{[2]} - \frac{1}{2}B^{[2]}.$$
\end{theorem}

The remainder of this subsection proves this assertion.

Recall that the image of $\varphi$ is a smooth quadric surface. For $Q$ a smooth quadric surface, $\mathrm{Pic}(Q)$ is generated by two classes $L_1, L_2$ with intersection pairing given by $$L_i \cdot L_j = \begin{cases}
    0&\mbox{ if }i=j\\
    1&\mbox{ if }i\neq j.
\end{cases}$$ \cite{Bertram-Coskun:2013} shows that $\Nef(Q^{[n]})$ is the convex cone in $N^1(Q^{[n]})$ spanned by $$\left\{L_1^{[n]}, L_2^{[n]}, \displaystyle\frac{B^{[n]}_Q}{2}\right\},$$ where $B^{[n]}_Q$ is the exceptional divisor of the Hilbert-Chow morphism $Q^{[n]}\to Q^{(n)}$. The strategy now is to use a rational map $\phi: X^{[2]} \dashrightarrow Q^{[2]}$ arising from the map $\varphi$ to deduce nefness of the candidate divisor on $X^{[2]}$ from nefness on $Q^{[2]}$. 

There are two classes of subschemes $\xi \subset X$ that arise as points of $X^{[2]}$: Pairs of distinct points $\{x_1, x_2\}$ and a length 2 subscheme supported at a single point $p$, which we think of as the pair $\{p, v\}$ where $v \in T_{X,p}$. Let $R$ be the ramification locus of $\varphi$. Define $\phi$ as follows:
\[
    \phi(\xi) = 
    \begin{cases}
        \{\varphi(x_1),\varphi(x_2)\} & \mathrm{supp}(\xi) = \{x_1,x_2\} \text{ and } \{x_1, x_2\} \neq \varphi^{-1}(q)\mbox{ }\forall\mbox{ } q \in Q \\
        \{\varphi(p), d\varphi(v)\} & \mathrm{supp}(\xi) = \{p\} \text{ and } p \not\in R.
    \end{cases}
\]
As defined, $\phi$ is a morphism when restricted to the open locus $U$ of subschemes which are not fibers of $\varphi$ and which are not supported entirely along the ramification curve $R$. Furthermore, since everything in sight is normal and $\mathrm{codim}_X(X \setminus U) = 2$, we can extend the pullback of any line bundle $L \in \mathrm{Pic}(Q^{[2]})$ from $U$ to all of $X^{[2]}$. Applying this construction to $\phi$, we get line bundles $\phi^*L_1^{[2]},\phi^*L_2^{[2]}, \phi^*\frac{B_Q}{2} \in \mathrm{Pic}(X^{[2]})$. Since both Hilbert schemes are smooth, by the construction above $\phi^*L_i^{[2]} \cong w_i^{[2]}$ and $$\phi^*\frac{B_Q}{2} \cong \frac{B_X}{2}.$$

In order to show that $\phi^*(L_1^{[2]} + L_2^{[2]} - \frac{B_Q}{2})$ is nef, we need to examine the base locus of $\phi$. Since $\varphi^{-1}(p) = \{q_1, q_2\} \subset X$ (not necessarily distinct), we have a flat family $\mathcal{X} \to Q$. This induces a rational map $\varphi': Q \dashrightarrow X^{[2]}$. By the universal property of the Hilbert scheme, $\varphi'$ is a morphism, and $Q' \coloneq \varphi'(Q) \cong Q$. Similarly, let $R \subset X$ be the locus where $d\varphi: T_X \to \varphi^*T_Q$ drops rank. Note that $d\varphi$ drops rank by exactly 1 and never 2, since the condition of dropping rank is a condition on the rank of $d\varphi$, and the expected dimension of the corank 2 locus is at most $0$. By Riemann-Hurwitz, $R \in |2w_1 + 2w_2|$ Note that points in the total space of $T_X|_R \to R$ are identified with the length 2 subschemes of $X$ whose supports are a single point of $R$. Set $S \coloneq \mathbb{P} T_X|_R$. By the same argument as above, we get a map $(d\varphi)': S \to X^{[2]}$. Thus, 
\[
\mathrm{Bs}(\phi) \subset S \cup Q'.
\]
(In fact, one can show via a local computation that the base locus is exactly the union of these two surfaces).
\\
\indent We analyze the components independently, beginning with $S$. This locus is a ruled surface over the curve $R$, so $\mathrm{Pic}(S)$ is generated by two elements $E$ and $F$, where $F$ is the class of a fiber of the ruled surface over a point of $R$ and $E$ is a section. First, observe that for $r \in R$ a closed point, the fiber $S_r$ corresponds to the distinct tangent directions of $X$ at $r$. It follows that $[F]$ is the class of a curve in $X^{[2]}$ which is contained entirely in $B_X$, the exceptional locus of the Hilbert-Chow morphism. In conclusion,
\begin{equation} \label{fiber1}
    F \cdot B_X = -2
\end{equation}

We also have the following theorem of Miyaoka:

\begin{lemma}{\cite[Lemma 2.1]{Fulger}} \label{miyaoka}
    Let $V$ be a vector bundle of slope $\mu$ on a complex projective curve $C$ with Harder-Narasimhan filtration $V = V_0 \supset V_1 \supset \dots \supset V_l$. Set $Q_i \coloneq V_{i-1}/V_i$, and $\mu_i \coloneq \mu(Q_i)$. Let $F$ be the class of a fiber of the structure morphism of $\mathbb{P}_C V$. Then $\mathrm{Nef}(\mathbb{P}_C V) = \langle c_1(\mathcal{O}_{\mathbb{P}_C V}(1)) - \mu_1 F, F \rangle$.
\end{lemma}

Recall from \cite{Huybrechts:2016} that the tangent bundle $T_X$ to a K3 surface $X$ has degree 0 and is slope stable with respect to any choice of ample class $H \in \mathrm{Pic}(X)$. Thus, $\mathrm{deg}_H(T_X) = 0$ and $\mu_H(T_X) = 0$.
\begin{lemma}{\cite[Theorem 2.1]{DH}}
    Let $G$ be a semistable vector bundle on a K3 surface $X$ with $\mathrm{rk}(G) = 2$ and $\mathrm{deg}_H(G) = 0$ for an ample line bundle $H \in \mathrm{Pic}(X)$. Then $G|_C$ is semistable for general nonhyperelliptic $C \in |H|$.
\end{lemma}
\indent Taking $H = 2(w_1 + w_2)$ and noting that $H$ is nonhyperelliptic, we can choose $R$ to be such that $T_X|_R$ is semistable of degree 0. With this, Miyaoka's theorem then says that $$\mathrm{Nef}(\mathbb{P} (T_X|_R)) = \langle c_1(\mathcal{O}_{\mathbb{P} (T_X|_R)}(1)), F \rangle.$$

\indent The other part of the base locus is the quadric $Q' \cong Q$. We know that $\mathrm{Pic}(Q') = \mathbb{Z} H_1 \oplus \mathbb{Z} H_2$. Let us compute the restriction of the generators of $\mathrm{Pic}(X^{[2]})$ to $Q'$. A representative of $[w_1^{[2]}]$ is given by all of the subschemes of length $2$ that meet a fixed general curve of class $w_1$. Since $\mathrm{Pic}(Q)$ is generated by the two lines $L_i$ and $\varphi^{-1}(L_i) = w_i$, we see that $H_i \cap w_j^{[2]} \neq 0$ when $i \neq j$ and is $0$ when $i = j$, since $w_i^2 = 0$ on $X$. In fact,
\[
    H_i \cap w_j^{[2]} = 
    \begin{cases}
       2 & i\neq j \\
       0 & i = j.
    \end{cases}
\]

Similarly, we know that the restriction of the nonreduced divisor $B_X|_{Q'}$ satisfies $B_X|_{Q'} \cdot H_i = 4$ for $i=1,2$. In summary, by counting points on closed subschemes and ramification of curves, we get the following result.

\begin{lemma}
    We have the following intersection numbers.
    $$w_1^{[2]}|_{Q'} \cdot H_1 = 0,\,\,w_1^{[2]}|_{Q'} \cdot H_2 = 2,\,\,w_2^{[2]}|_{Q'} \cdot H_1 = 2,\,\,w_2^{[2]}|_{Q'} \cdot H_2 = 0,\,\,B|_{Q'} \cdot H_1 = 4,\,\,B|_{Q'} \cdot H_2 = 4.$$
\end{lemma}

\begin{corollary} \label{Q_prime_res}
    We have the following equalities in $\Pic(Q')$. $$w_1^{[2]}|_{Q'} = 2H_1,\,\,w_2^{[2]}|_{Q'} = 2H_2,\,\,B|_{Q'} = -4(H_1 + H_2).$$
\end{corollary}

Recall the theorem of Fujita (see \cite[Theorem 2.8]{Fujita}), which says that if a line bundle $\mathcal{O}(D)$ on a projective variety restricts to an ample line bundle on the base locus $\mathrm{Bs}(|D|)$, then $\mathcal{O}(D)$ is semiample, i.e. $\mathcal{O}(mD)$ is basepoint free for $m\gg 0$. Since a semiample line bundle generates a ray of the nef cone by definition, it is sufficient to know that the candidate nef class $$\Gamma=\phi^*\left(L_1^{[2]} + L_1^{[2]} - \frac{B_Q}{2}\right)$$ is semiample.

\begin{proposition}
    Let $\Gamma_t$ denote the class $$\phi^*\left(L_1^{[2]} + L_1^{[2]} - (1-t)\frac{B_Q}{2}\right).$$ Then $\Gamma_t$ is semiample for $0 < t < 2$.
\end{proposition}
\begin{proof}
    By applying the construction about pullbacks of line bundles, we have $\phi^*L_i^{[2]} = w_i^{[2]}$ and $\phi^*B_Q = B_X$. Away from the base locus $$\mathrm{Bs}(\phi) = Q' \cup S,$$ $\Gamma_t$ is nef. Now, we first deal with $S$. Set $E = c_1(\mathcal{O}_{\mathbb{P} T_X|_R}(1))$. By \ref{fiber1}, $B_X|_S \cdot F = -2$. Since the restrictions to $S$ of $w_i^{[2]}$ are nef (see \cite{Laz}), $\Gamma_t$ is nef. Since it is the sum of nef classes and is not a multiple of $F$, it cannot lie on the edge of the $\mathrm{Nef}(S)$. So $\Gamma_t|_S$ is ample. For $Q'$, by the intersection numbers computed above in Corollary \ref{Q_prime_res}, we have $\Gamma_t \in \mathrm{Nef}(Q')$ and for $0 < t < 2$, $\Gamma_t$ is ample. So $\Gamma_t$ restricts to an ample line bundle on the entire base locus. By Fujita's theorem, we are done.
\end{proof}

Now that we have constructed a set of semiample (and therefore nef) line bundles, we use the fact that the nef cone is closed to finish the proof of Theorem \ref{thm:Nef_cone_nk2}.

\begin{proposition}
    $\Gamma$ is nef on $X^{[2]}$.
\end{proposition}
\begin{proof}
    Consider the sequence $\{\Gamma_{\frac{1}{n}}\}$, $n \in \mathbb{N}$. The limit of the sequence is $\Gamma$, and by the previous proposition, each element of the sequence is nef. Since the nef cone is closed, $\Gamma$ is also nef.
\end{proof}

\begin{remark}
    It is interesting to note that the extremality of the final ray shown above is already suggested by the above proof that the class is nef. We showed that the class was nef by exhibiting it as a limit. Indeed, it is exactly the border case in the Fujita theorem: its restriction to $\mathrm{Bs}(\phi)$ actually fails to be ample (via the intersection numbers computed above), but is the limit of a sequence whose elements are. It would be interesting to examine the relationship between the boundary cases in this application of Fujita's theorem and extremality of the associated classes on the Hilbert scheme.
\end{remark}

\subsection{A direct computation of the nef cone of $X^{[2]}$}\label{subsec:nef_cone_nk2_lattice_computation}

In this subsection, we use \cite[Theorem 12.1]{Bayer-Macri-K3} to illustrate an alternate way of calculating the nef cone of $X^{[2]}$ in case I, assuming $k=2$.

Let us first recall some notation. Let $H^*_{\operatorname{alg}}(X,\mathbb{Z})$ be the algebraic Mukai lattice of $X$ with Mukai pairing denoted by $\langle-,-\rangle$. The Mukai vector $\mathbf{v}:K(X)\to H^*_{\operatorname{alg}}(X,\mathbb{Z})$ is defined by
$$\mathbf{v}(E)=\ch(E)\sqrt{\operatorname{td}(X)}=(\ch_0,\ch_1,\ch_0+\ch_2),$$ and the Mukai pairing takes the form $\langle(r,c,s),(r',c',s')\rangle=(c,c')-rs'-r's$, where $(-,-)$ is the intersection pairing on $H^2(X,\mathbb{Z})$. Let ${\bf v}:=(1,0,1-2)=(1,0,-1)$ be the Mukai vector corresponding to the moduli space $X^{[2]}$. In other words, $X^{[2]}$ is the moduli space $M_{\sigma}({\bf v})$ of Bridgeland semistable objects with Mukai vector ${\bf v}$, where $\sigma$ is a Bridgeland stability condition lying above the Gieseker wall. Given a Mukai vector ${\bf w}$, we denote $\langle{\bf w},{\bf w}\rangle$ by ${\bf w}^2$, and its orthogonal complement by $${\bf w}^{\perp}:=\left\{{\bf w}'\in H^*_{\operatorname{alg}}(X,\mathbb{Z})\,|\,\langle{\bf w},{\bf w}'\rangle=0\right\}.$$ The closed cone of positive classes defined by the Beauville–Bogomolov quadratic form on the N\'eron-Severi group of the hyperk\"ahler manifold $X^{[2]}$ is denoted by $\overline{\operatorname{Pos}}(X^{[2]})$. Finally, let $\theta:{\bf v}^{\perp}\to NS(X^{[2]})$ be the isometry of lattices obtained from the Hodge-structure isomorphism between the orthogonal complement of ${\bf v}$ in $H^*(X,\mathbb{Z})$ and $H^2(X^{[2]},\mathbb{Z})$ (see \cite{Yoshioka01}). We refer the reader to \cite{Bayer-Macri-K3,NuerYoshioka} for more details.

By \cite[Theorem 12.1]{Bayer-Macri-K3}, the nef cone of $X^{[2]}$ is cut out in $\overline{\operatorname{Pos}}(X^{[2]})$ by all linear subspaces of the form $\theta({\bf v}^{\perp}\cap{\bf a}^{\perp})$, for all classes ${\bf a}^{\perp}$ in $H^*_{\operatorname{alg}}(X,\mathbb{Z})$ satisfying the inequalities ${\bf v}^2\geq-2$ and $0\leq\langle{\bf v},{\bf a}\rangle\leq{\bf v}^2/2$.

By the proof of Theorem \ref{thm:Nef_cone_nk2} above, we know that all the three divisor classes $w_1^{[2]}, w_2^{[2]}$ and $$w_1^{[2]} + w_2^{[2]} - \frac{1}{2}B^{[2]}$$ are nef on $X^{[2]}$. Thus, to prove that the nef cone of $X^{[2]}$ is spanned by $w_1^{[2]}, w_2^{[2]}$ and $$w_1^{[2]} + w_2^{[2]} - \frac{1}{2}B^{[2]}$$ (i.e., that it is not larger than the cone spanned by these vectors), it suffices to show that each of the three faces of the cone spanned by these three vectors is among the linear spaces $\theta({\bf v}^{\perp}\cap{\bf a}^{\perp})$ mentioned above. Since $\theta((0,-w_1,0))=w_1^{[2]}$, $\theta((0,-w_2,0))=w_2^{[2]}$, and $\theta((1,0,n-1))=-\tfrac{B^{[2]}}{2}$ (see \cite[Remark 10.4]{Bayer-Macri-PositivityLemma}), it follows from some straightforward computations that the Mukai vectors ${\bf a}=(1,-w_1,1)$, $(1,-w_2,1)$ and $(1,0,1)$ give us the desired hyperplanes.

\subsection{Some closing remarks and possible future directions}
\indent It is interesting to note that the classical flavored proof for two points on one class of K3 surfaces given above only used one K3-specific ingredient. It is obviously necessary to understand the intersection theory of the surface, but that is understood in many cases. The K3-specific ingredient was stability of the restricted tangent bundle. This indicates that a similar technique might be attempted for surfaces that are not K3s. Furthermore, even the stability of the restricted tangent bundle on the nose was not totally necessary. The full version of the theorem cited from \cite{Fulger} uses only the slope of the first graded piece of the Harder-Narasimhan filtration. As such, it should be possible to achieve similar results for other surfaces at the expense of a possibly more difficult computation. In the case that this computation is unmanageable at present, it would still be interesting to compare $T_X$ and $f^*T_Y$ for any finite map $X \to Y$ of surfaces to obtain a bound on positivity. Finally, working with a larger number $n$ of points should be possible by examining different vector bundles that still have some relation to ramification. Each separate $n$ requires a different computation using this method. On the other hand, the brilliance of the methods pioneered in \cite{Bayer-Macri-K3}, \cite{Bayer-Macri-PositivityLemma} shows in how efficiently every value of $n$ is dealt with.

\bibliographystyle{alpha}
\bibliography{mainbib}

\newcommand{\etalchar}[1]{$^{#1}$}
\begin{thebibliography}{ABCH13}

\bibitem[AB13]{Arcara-Bertram:2013}
Daniele Arcara and Aaron Bertram.
\newblock Bridgeland-stable moduli spaces for {{\(K\)}}-trivial surfaces.
\newblock {\em J. Eur. Math. Soc. (JEMS)}, 15(1):1--38, 2013.

\bibitem[ABCH13]{ABCH13}
Daniele Arcara, Aaron Bertram, Izzet Coskun, and Jack Huizenga.
\newblock The minimal model program for the {H}ilbert scheme of points on
  {$\mathbb{P}^2$} and {B}ridgeland stability.
\newblock {\em Adv. Math.}, 235:580--626, 2013.

\bibitem[ADHL15]{Coxrings:2015}
Ivan Arzhantsev, Ulrich Derenthal, J{\"u}rgen Hausen, and Antonio Laface.
\newblock {\em Cox rings}, volume 144 of {\em Camb. Stud. Adv. Math.}
\newblock Cambridge: Cambridge University Press, 2015.

\bibitem[BC13]{Bertram-Coskun:2013}
Aaron Bertram and Izzet Coskun.
\newblock The birational geometry of the {Hilbert} scheme of points on
  surfaces.
\newblock In {\em Birational geometry, rational curves, and arithmetic. Based
  on the symposium ``Geometry over closed fields'', St. John, UK, February
  2012}, pages 15--55. New York, NY: Springer, 2013.

\bibitem[BHL{\etalchar{+}}16]{Bol+:2016}
Barbara Bolognese, Jack Huizenga, Yinbang Lin, Eric Riedl, Benjamin Schmidt,
  Matthew Woolf, and Xiaolei Zhao.
\newblock Nef cones of {Hilbert} schemes of points on surfaces.
\newblock {\em Algebra Number Theory}, 10(4):907--930, 2016.

\bibitem[BM14a]{Bayer-Macri-K3}
Arend Bayer and Emanuele Macr\`{i}.
\newblock M{MP} for moduli of sheaves on {K}3s via wall-crossing: nef and
  movable cones, {L}agrangian fibrations.
\newblock {\em Invent. Math.}, 198(3):505--590, 2014.

\bibitem[BM14b]{Bayer-Macri-PositivityLemma}
Arend Bayer and Emanuele Macr\`{i}.
\newblock Projectivity and birational geometry of {B}ridgeland moduli spaces.
\newblock {\em J. Amer. Math. Soc.}, 27(3):707--752, 2014.

\bibitem[Bri07]{BriOG:2007}
Tom Bridgeland.
\newblock Stability conditions on triangulated categories.
\newblock {\em Ann. Math. (2)}, 166(2):317--345, 2007.

\bibitem[Bri08]{BriK3:2008}
Tom Bridgeland.
\newblock Stability conditions on {{\(K3\)}} surfaces.
\newblock {\em Duke Math. J.}, 141(2):241--291, 2008.

\bibitem[CG90]{Catanese-Goettsche:1990}
Fabrizio Catanese and Lothar G{\oe}ttsche.
\newblock {$d$}-very-ample line bundles and embeddings of {H}ilbert schemes of
  {$0$}-cycles.
\newblock {\em Manuscripta Math.}, 68(3):337--341, 1990.

\bibitem[Che98]{Cheah:1998}
Jan Cheah.
\newblock Cellular decompositions for nested {H}ilbert schemes of points.
\newblock {\em Pacific J. Math.}, 183(1):39--90, 1998.

\bibitem[DH22]{DH}
Yajnaseni Dutta and Daniel Huybrechts.
\newblock Maximal variation of curves on {{\(K3\)}} surfaces.
\newblock {\em Tunis. J. Math.}, 4(3):443--464, 2022.

\bibitem[Fog73]{Fogarty}
J.~Fogarty.
\newblock Algebraic families on an algebraic surface. {II}: {The} {Picard}
  scheme of the punctual {Hilbert} scheme.
\newblock {\em Am. J. Math.}, 95:660--687, 1973.

\bibitem[Fuj83]{Fujita}
Takao Fujita.
\newblock Semipositive line bundles.
\newblock {\em J. Fac. Sci., Univ. Tokyo, Sect. I A}, 30:353--378, 1983.

\bibitem[Ful11]{Fulger}
Mihai Fulger.
\newblock The cones of effective cycles on projective bundles over curves.
\newblock {\em Math. Z.}, 269(1-2):449--459, 2011.

\bibitem[Huy16]{Huybrechts:2016}
Daniel Huybrechts.
\newblock {\em Lectures on {{\(K\)}}3 surfaces}, volume 158 of {\em Camb. Stud.
  Adv. Math.}
\newblock Cambridge: Cambridge University Press, 2016.

\bibitem[Knu01]{Knutsen:2001}
Andreas~Leopold Knutsen.
\newblock On {$k$}th-order embeddings of {$K3$} surfaces and {E}nriques
  surfaces.
\newblock {\em Manuscripta Math.}, 104(2):211--237, 2001.

\bibitem[Kop21]{Kopper:21}
John Kopper.
\newblock The nef cone of the {H}ilbert scheme of points on rational elliptic
  surfaces and the cone conjecture.
\newblock {\em Canad. Math. Bull.}, 64(1):216--227, 2021.

\bibitem[Laz04]{Laz}
Robert Lazarsfeld.
\newblock {\em Positivity in algebraic geometry. {I}. {Classical} setting: line
  bundles and linear series}, volume~48 of {\em Ergeb. Math. Grenzgeb., 3.
  Folge}.
\newblock Berlin: Springer, 2004.

\bibitem[LZ19]{LiZhao19}
Chunyi Li and Xiaolei Zhao.
\newblock Birational models of moduli spaces of coherent sheaves on the
  projective plane.
\newblock {\em Geom. Topol.}, 23(1):347--426, 2019.

\bibitem[Mac14]{Maciocia:2014}
Antony Maciocia.
\newblock Computing the walls associated to {Bridgeland} stability conditions
  on projective surfaces.
\newblock {\em Asian J. Math.}, 18(2):263--280, 2014.

\bibitem[MS17]{Macri-Schmidt:2017}
Emanuele Macr{\`{\i}} and Benjamin Schmidt.
\newblock Lectures on {Bridgeland} stability.
\newblock In {\em Moduli of curves. CIMAT Guanajuato, Mexico 2016. Lecture
  notes of a CIMPA-ICTP school, Guanajuato, Mexico, February 22 -- March 4,
  2016}, pages 139--211. Cham: Springer, 2017.

\bibitem[Nue16]{Nuer}
Howard Nuer.
\newblock Projectivity and birational geometry of {Bridgeland} moduli spaces on
  an {Enriques} surface.
\newblock {\em Proc. Lond. Math. Soc. (3)}, 113(3):345--386, 2016.

\bibitem[NY20]{NuerYoshioka}
Howard Nuer and K{\={o}}ta Yoshioka.
\newblock M{MP} via wall-crossing for moduli spaces of stable sheaves on an
  {E}nriques surface.
\newblock {\em Adv. Math.}, 372:107283, 119, 2020.

\bibitem[Rah24]{Raha2024}
Neelarnab Raha.
\newblock Ample cones of {Hilbert} schemes of points on hypersurfaces in
  {{\(\mathbb{P}^3\)}}.
\newblock {\em Proc. Am. Math. Soc.}, 152(12):5067--5081, 2024.

\bibitem[RY20]{Ryan-Yang:2020}
Tim Ryan and Ruijie Yang.
\newblock Nef cones of nested {H}ilbert schemes of points on surfaces.
\newblock {\em Int. Math. Res. Not. IMRN}, (11):3260--3294, 2020.

\bibitem[Yos01]{Yoshioka01}
K{\={o}}ta Yoshioka.
\newblock Moduli spaces of stable sheaves on abelian surfaces.
\newblock {\em Math. Ann.}, 321(4):817--884, 2001.

\bibitem[Yos16]{Yoshioka16}
K{\={o}}ta Yoshioka.
\newblock Bridgeland's stability and the positive cone of the moduli spaces of
  stable objects on an abelian surface.
\newblock {\em Adv. Stud. Pure Math.}, 69:473--537, 2016.

\bibitem[YY14]{YY}
Shintarou Yanagida and K{\={o}}ta Yoshioka.
\newblock Bridgeland's stabilities on abelian surfaces.
\newblock {\em Math. Z.}, 276(1-2):571--610, 2014.

\end{thebibliography}

\end{document}